\documentclass[10pt,reqno]{amsart}
  \usepackage{geometry}
    \usepackage{amsxtra}
    \usepackage{amsmath}
\usepackage{amssymb}
\usepackage{lscape}
\usepackage{pdflscape,lipsum}
\usepackage{etoolbox}
\usepackage{color}
\usepackage{multicol}
\usepackage{setspace}
\usepackage{multirow}
\usepackage{enumerate}
\newcommand{\printvalues}{topsep=\the\topsep; itemsep=\the\itemsep; parsep=\the\parsep; partopsep=\the\partopsep}

\newcommand{\alt}{\raise1pt\hbox{$\bigwedge$}}

\newcommand\s{\sigma}

\newcommand\la{\langle}
\newcommand\ra{\rangle}
\newcommand{\ip}{\langle\cdot,\cdot\rangle}

\newcommand{\on}{\operatorname} 
\newcommand\dd{{\mathfrak d}}

\newcommand\ggo{{\mathfrak g}}
\newcommand\aff{\mathfrak {aff}}

\newcommand\RR{\mathbb R}

\newcommand\sign{\operatorname{sign}}

\newcommand\Ric{\operatorname{Ric}}

\newcommand\hol{\operatorname{Hol}}
\newcommand\Hol{\operatorname{Hol}}
\newcommand\Span{\operatorname{Span}}
\theoremstyle{plain}

\newtheorem{thm}{Theorem}[section]
\newtheorem{lem}[thm]{Lemma}
\newtheorem{prop}[thm]{Proposition}
\newtheorem{cor}[thm]{Corollary}

\theoremstyle{definition}
\newtheorem{defn}[thm]{Definition}

\newtheorem*{Problem*}{Problem}

\theoremstyle{remark}
\newtheorem{rem}{Remark}[section]

\begin{document}

\title[{\tiny Parallel skew-symmetric tensors on $4$ dimensional metric Lie algebras}]{Parallel skew-symmetric tensors
 on $4$-dimensional metric Lie algebras}


\author{A. Herrera}
\address{Facultad de Ciencias Exactas, F\'isicas y Naturales, Universidad Nacional de C\'ordoba, Argentina. 
}
\curraddr{}
\email{cecilia.herrera@unc.edu.ar}
\thanks{The author was partially supported by CONICET, ANPCyT and SECyT-UNC (Argentina).}


\date{}

\dedicatory{}

\begin{abstract}
We give a complete classification, up to isometric isomorphism and scaling, of $4$-dimensional metric Lie algebras $(\ggo,\ip)$ that admit a non-zero parallel skew-symmetric endomorphism. In particular, we distinguish those metric Lie algebras that admit such an endomorphism which is not a multiple of a complex structure, and for each of them we obtain the de Rham decomposition of the associated simply connected Lie group with the corresponding left invariant metric. 
On the other hand, we find that the associated simply connected Lie group is irreducible as a Riemannian manifold for those metric Lie algebras where each parallel skew-symmetric endomorphism is a multiple of a complex structure.

\end{abstract}

\maketitle
\section{Introduction}

Let $(M,g)$ be a Riemannian manifold. A skew-symmetric $(1,1)$-tensor $H:TM\to TM$  is said to be {\em parallel} if $(\nabla_X H)Y=0 $ for all vector fields $X,Y\in\mathfrak{X} (M)$, where $\nabla$ denotes the Levi-Civita connection associated to $g$.
If in addition $H^2=-I$ (where $I$ denotes the identity map), then $H$ is a complex structure and $(M,g,H)$ is called a K\"ahler manifold, an object widely studied in the literature. 
Here, we are interested in (connected) manifolds that admit
parallel tensors $H$ that are not multiple of a complex structure, that is, $H^2\neq-\lambda^2 I$ for any $\lambda\in\mathbb{R}$. 
  
We focus particularly on pairs $(G,g)$, where $G$ is a 4-dimensional non-abelian Lie group and $g$ is a left invariant metric on $G$, and we search for left invariant parallel skew-symmetric $(1,1)$-tensors on $G$. As usual, we work at the Lie algebra level. Namely, we consider non-abelian 4-dimensional metric  Lie algebras $(\mathfrak{g},\langle\cdot,\cdot\rangle)$, and look for non-zero skew-symmetric endomorphisms $H:\mathfrak{g}\rightarrow\mathfrak{g}$ that are parallel; see Section \ref{Section: Preliminaries}.
If we add the condition $H^2=-I$, then this problem was completely settled by Ovando in \cite{Ov}, where  $4$-dimensional pseudo-K\"ahler Lie algebras have been classified up to equivalence. 
In the present paper, we classify up to isometric isomorphism and scaling non-abelian  4-dimensional metric Lie algebras $(\ggo,\ip)$ admitting a parallel endomorphism that is not a multiple of a complex structure. In addition, given such a pair  $(\ggo,\ip)$, we classify all possible parallel endomorphisms. These results are included in Section \ref{Section: Main theorem}; more precisely, in Theorem \ref{dim4} and Proposition \ref{equivalencias1}.

Finally, in Section \ref{The De Rham decomposition}, for each non-abelian $4$-dimensional metric Lie algebra $(\ggo,\ip)$ admitting a non-zero parallel skew-symmetric endomorphism, we study the de Rham decomposition of the associated simply connected Riemannian Lie group $(G,g)$. 
We find that $(G,g)$ is irreducible (as a Riemannian manifold) if and only if the only parallel endomorphisms on $(\ggo,\ip)$  are multiple of complex structures. Note that the `only if' part is a consequence of the following well-known result:
{\em Let $(M,g)$ be a complete simply connected irreducible Riemannian manifold. Then every parallel skew-symmetric  (1,1)-tensor is a multiple of a complex structure}; \cite[Theorem 10.3.2]{Petersen}.


\ 

\noindent
 { \it Acknowledgment:} This article grew out of part of my PhD thesis (\cite{ACHerrera}) under the supervision of  Isabel Dotti. I would like to express my thanks to her and Adri\'an Andrada for their remarks, corrections and useful suggestions in the preparation of this work.

\

\section{Preliminaries}\label{Section: Preliminaries}
In this section we recall the general definition of parallel skew-symmetric tensor on a Riemannian manifold and then we adapt it to the case of a left invariant parallel skew-symmetric tensor on a Lie group with a left invariant metric. This enables us to work at the Lie algebra level. We also define a notion of equivalence on parallel tensors on a metric Lie algebra.

\subsection{Parallel tensors} Let $(M,g)$ be a Riemannian manifold and let $\nabla$ be the associated Levi-Civita connection. 
A skew-symmetric tensor $H:TM\to TM$ is said to be parallel if $\left(\nabla_XH\right)Y=0$ for all  $X, Y \in\mathfrak{X}(M)$, where $\nabla H$ denotes the covariant derivative of $H$. Recall that
\[ \left( \nabla_{X}H\right)Y = \nabla _X \left( HY\right) - H
\left(\nabla _X Y \right).\]
If in addition $H^2=-I$, then $H$ is a complex structure and $(M,g,H)$ is a K\"ahler manifold, with the K\"ahler form given by $\omega(X,Y)=g(HX,Y)$.

 
\subsection{Left invariant parallel tensors on Lie groups}
Let $G$ be a Lie group, and let $\ggo$ be the Lie algebra of left invariant vector fields on $G$. 
We assume that $g$ is a left invariant metric, i.e., the left translation $L_p:G\rightarrow G$ is an  isometry for any $p\in G$. 
Every left invariant metric $g$ on $G$ determines an inner product $\la\cdot,\cdot\ra$ on $\ggo$: $\langle x,y\rangle:=g_e(x_e,y_e)$ for $x,y\in\ggo$; and conversely, any inner product on $\ggo$ determines uniquely a left invariant metric on $G$.

We fix a left invariant metric $g$ on $G$ and denote as $\la\cdot,\cdot\ra$ its induced inner product on $\ggo$. Let $\nabla$ be the Levi-Civita connection associated to $g$.
It is a fact that for $x,y\in\ggo$, $\nabla_xy\in\ggo$ and it is given by the Koszul formula 
\begin{equation}\label{koszul}
2\la \nabla_xy,z\ra=\la [x,y],z\ra-\la [y,z],x\ra+\la [z,x],y\ra.
\end{equation}
It is easy to see that $\nabla_x:\ggo\to \ggo$ is a skew-symmetric endomorphism with respect to $\la \cdot, \cdot \ra$ for any $x\in\ggo$.

We consider skew-symmetric $(1,1)$-tensors $H:TG\to TG$ that are left invariant: $(\mathrm{d} L_p)_{q}H_q=H_{pq}(\mathrm{d}L_p)_q$ for all $p,q\in G$. Every such tensor induces a skew-symmetric endomorphism $H:\ggo\to\ggo$ (again denoted by $H$) and conversely any skew-symmetric endomorphism on $\ggo$  extends uniquely to a skew-symmetric left invariant tensor on $G$.

It is easy to see that $H:TG\to TG$ as above is parallel if and only if the associated endomorphism $H:\ggo\to\ggo$ is parallel, in the sense that $H$ commutes with $\nabla_x:\ggo\to\ggo$ for all $x\in\ggo$.

\subsection{Parallel endomorphisms on metric Lie algebras}
The previous discussion enables us to work algebraically. Namely, we fix an abstract real Lie algebra $\ggo$ endowed with an inner product $\ip$, and we search for skew-symmetric linear endomorphisms $H:\ggo\to \ggo$ satisfying the condition $\nabla_x(Hy)=H\nabla_xy$ for all $x,y\in\ggo$, where $\nabla_x y$ is defined via the Koszul formula (\ref{koszul}). 

The pair $(\ggo,\la\cdot,\cdot\ra)$ is called a metric Lie algebra, and an endomorphism $H$ as above is called a parallel tensor. If in addition $H^2=-I$, then $(\ggo,\ip)$ is called a K\"ahler Lie algebra.

We intend to classify triples $(\ggo,\la\cdot,\cdot\ra,H)$, where $H$ is a non-zero parallel skew-symmetric  endomorphism.  
Note, that if $H$ is a parallel skew-symmetric endomorphism on $(\ggo,\ip)$, then the same holds on $(\ggo,\lambda^2\ip)$ for any $\lambda>0$. This implies that we may consider the same parallel skew-symmetric endomorphism $H$ on homothetic metrics. In addition, $cH$ will be also parallel on $(\ggo,\ip)$ for any $c\in \RR$.

A natural notion of equivalence on the set of parallel skew-symmetric endomorphisms on $(\ggo, \la\cdot,\cdot\ra)$ is given by the following definition:
\begin{defn}
\label{equivalentes}
$H_1$ and $H_2$ are said to be equivalent if there exists an isometric isomorphism of Lie algebras  $\phi:\mathfrak{g}\rightarrow\mathfrak{g}$ such that $\phi H_1=H_2 \phi$. 
\end{defn}

We can now restate the classification problem as follows:
\begin{Problem*} \label{Classification problem} Classify $4$-dimensional metric Lie algebras $(\ggo, \ip)$ (up to isometry and homothety) that admit a non-zero parallel skew-symmetric endomorphism $H$, and classify  $H$ up to equivalence. 
\end{Problem*}

\

\section{Main results}\label{Section: Main theorem}
In this section we determine all triples $(\ggo,\ip,H)$, where $(\ggo,\ip)$ is a non-abelian $4$-dimensional metric Lie algebra and $H:\ggo\to\ggo$ is a non-zero parallel skew-symmetric endomorphism. The case where $H^2=-I$  was done in \cite{Ov} and we start by recalling this. 

\subsection{4-dimensional K\"ahler Lie algebras}

We first list the non-abelian real solvable Lie algebras of dimension $\leq 3$ with the notation used in \cite{ABDO}. These are:
\begin{align}
\begin{split}
\label{3dimensional}
\mathfrak{aff}(\mathbb{R}):& \quad  [e_1,e_2]=e_2,\\
\mathfrak{h}_3:& \quad [e_1,e_2]=e_3,\\
\mathfrak{r}_{3,\lambda}:& \quad [e_1,e_2]=e_2,\  [e_1,e_3]=\lambda e_3,\ \ \lambda\in\mathbb{R},\\
\mathfrak{r}'_{3,\lambda}:& \quad [e_1,e_2]=\lambda e_2-e_3,\ [e_1,e_3]=e_2+\lambda e_3,\ \lambda\in\mathbb{R},
\end{split}
\end{align}
where $\{ e_1,e_2\}$ is a basis of $\aff(\RR)$ and  $\{e_1,e_2,e_3\}$ is a basis of $\mathfrak{h}_3$, $\mathfrak{r}_{3,\lambda}$ and $\mathfrak{r}'_{3,\lambda}$.  
They are all pairwise non-isomorphic, except for $\mathfrak{r}_{3,\lambda}\cong \mathfrak{r}_{3,1/\lambda}$ if $\lambda\neq 0$, and $\mathfrak{r}'_{3,\lambda}\cong\mathfrak{r}'_{3,-\lambda}$.

As a side note, $\mathfrak{r}_{3,-1}$ is the Lie algebra of the group of rigid motions on the Minkowski plane, and it is usually denoted $\mathfrak{e}(1,1)$. Note also that  $\mathfrak{r}_{3,0}=\mathbb{R} \times \mathfrak{aff}(\mathbb{R})$, and that $\mathfrak{r}'_{3,0}$ is the Lie algebra of the group of rigid motions on the 2-dimensional euclidean plane, and is usually denoted $\mathfrak{e}(2)$.

We now list three families of 4-dimensional solvable Lie algebras expressed in the basis $\{ e_1,e_2,e_3,e_4\}$.  We follow the notation given in \cite{ABDO}.
\begin{align}
\begin{split}
\label{notationused}
 \mathfrak{r}'_{4,\lambda,0}& : [e_4,e_1]=\lambda e_1,\ [e_4,e_2]=-e_3, \ [e_4,e_3]=e_2, \ \lambda>0,\\
\mathfrak{d}_{4,\lambda}&:[e_4,e_1]=\lambda e_1, \ [e_4,e_2]=(1-\lambda)e_2, \ [e_4,e_3]=e_3, \ [e_1,e_2]=e_3,\  \lambda \geq \frac{1}{2},\\
\mathfrak{d}'_{4,\lambda}&:[e_4,e_1]=\lambda e_1-e_2, \ [e_4,e_2]=e_1+\lambda e_2, \ [e_4,e_3]=2 \lambda e_3, \ [e_1,e_2]=e_3,\  \lambda \geq 0.
\end{split}
\end{align}
These are all pairwise non-isomorphic, according to \cite[Theorem 1.5]{ABDO}.

The content of the next theorem is included in  \cite[Proposition 3.3]{Ov}, where $4$-dimensional pseudo-K\"ahler Lie algebras were classified. Here we only consider positive definite K\"ahler Lie algebras.

\begin{thm}\cite[Proposition 3.3]{Ov}
Let $(\ggo,\ip,J)$ be a $4$-dimensional 
K\"ahler  Lie algebra. Then there exists an orthonormal basis $\{ e_1,e_2,e_3,e_4\}$ where the Lie brackets and $J$ are given as in Table 1.
\end{thm}

     \begin{table}[h!]
\centering
\begin{spacing}{1.5}
\begin{tabular}{l l l }
\hline
\hline

Lie algebra  & Lie bracket in an orthonormal basis & Complex structure \\  \hline
     \hline
$\RR^2 \times \aff(\RR)$& $[e_1,e_2]=t e_2$, $ t>0$ & $Je_1=e_2$, $Je_3=e_4$ 
  \\ \hline

$ \RR \times \mathfrak{e}(2)$ &$ [e_1,e_2]=-t e_3$, $[e_1,e_3]=t e_2$, $  t>0$ & $Je_1=e_4$, $Je_2=e_3$  \\
  \hline
$\mathfrak{r}_{4,\lambda,0}'$  &$ [e_4,e_1]=t e_1$, $[e_4,e_2]=-\frac{t}{\lambda} e_3$,   &  $J_1e_1=-e_4$, $J_1e_2=e_3$  \\
      $\lambda >0$                         &$ [e_4,e_3]=\frac{t}{\lambda} e_2$ , $t >0$ &  $J_2e_1=-e_4$, $J_2e_2=-e_3$  \\ 
\hline
$\aff(\RR)\times \aff(\RR)$&$[e_1,e_2]=te_2$, $ [e_3,e_4]=s e_4 $,    $t, s>0$ & $Je_1=e_2$, $Je_3=e_4$ \\                                
 \hline

$\mathfrak{d}_{4,\frac{1}{2}}$& $ [e_1,e_2]=t e_3$, $[e_4,e_3]=t e_3$, & $Je_1=e_2$, $Je_4=e_3$
    \\
   &  $ [e_4,e_1]=\frac{t}{2} e_1$, $[e_4,e_2]=\frac{t}{2} e_2$, $t>0$&  \\ \hline
   $\mathfrak{d}_{4,2}$ & $[e_1,e_2]=t e_3$, $[e_4,e_3]=\frac{t}{2} e_3$,   & $Je_4=-e_1$, $Je_2=e_3$ \\ 
&$ [e_4,e_1]=t e_1$, $[e_4,e_2]=-\frac{t}{2} e_2$, $t >0$ & \\ \hline
 
$\mathfrak{d}_{4,\frac{\delta}{2}}'$ & $[e_1,e_2]=t e_3$, $[e_4,e_1]=\frac{t}{2}e_1-\frac{t}{\delta}e_2 $, & $J_1e_1=e_2$, $J_1e_4=e_3$ \\
$\delta>0$   & $  [e_4,e_3]=t e_3$, $[e_4,e_2]=\frac{t}{\delta}e_1+\frac{t}{2}e_2$, $t>0$  &  $J_2e_1=-e_2$, $J_2e_4=-e_3$  \\
   \hline     
\end{tabular}
\end{spacing}
\caption{ $4$-dimensional K\"ahler Lie algebras}
   \label{tabla11}
  \end{table}

\subsection{Parallel endomorphisms that are not multiple of a complex structure}
 We now describe non-abelian $4$-dimensional metric Lie algebras that admit a parallel skew-symmetric endomorphism $H$ which is not multiple of a complex structure. We may assume that $H$ is always non-zero in this subsection. 

We will use the following lemma whose proof is straightforward.
\begin{lem}\label{lemautil}
 Let $A, B\in\mathbb{R}^{4\times 4}$ be two skew-symmetric matrices such that \[B=\left[ \begin{matrix}0&-s&0&0\\s&0&0&0\\0&0&0&-t\\0&0&t&0 \end{matrix}\right] \quad \mbox{ with } |s|\neq |t|.\]  
If $AB=BA$, then $A$ has also the form of $B$, that is, there exist $s',t' \in \mathbb{R}$ such that 
\[ A= \left[ \begin{matrix}0&-s'&0&0\\s'&0&0&0\\0&0&0&-t'\\0&0&t'&0 \end{matrix}\right].\]
\end{lem} 

\begin{thm}\label{dim4}
Let $(\ggo,\ip)$ be a non-abelian $4$-dimensional metric Lie algebra, and let $H:\ggo\to\ggo$ be a parallel skew-symmetric endomorphism which is not a multiple of a complex structure. Then there exists an orthogonal basis $\{e_1,f_1,e_2,f_2\}$ under which the Lie brackets, $\ip$ and $H$ are given as in Table \ref{tabla1} with $|a_1|\neq |a_2|$. 
Moreover, any endormophism appearing in Table \ref{tabla1} is parallel for the corresponding metric Lie algebra.


\end{thm}
\begin{table}[h!]
\centering
\begin{spacing}{1.25}
\begin{tabular}{l l l}
\hline
\hline
   Lie algebra  &  Metric  & Parallel tensor\\
\hline \hline
  {\small $\begin{matrix}\RR \times \mathfrak{e}(2):\\
   [e_1,e_2]=-f_2,\\ [e_1,f_2]= e_2   \end{matrix}$}
 &   {\small $\ip_t=\left[ \begin{matrix} t&&&\\ &t&&\\ &&t&\\ &&&t \end{matrix} \right]$, $t>0$} & {\small $\left[ \begin{array}{cccc} 0&-a_1&0&0\\ a_1&0&0&0\\ 0&0&0&-a_2\\ 0&0&a_2&0 \end{array} \right]$}
 \\
 \hline 
\begin{small}
 $
\begin{matrix} \RR^2\times \aff(\RR) :\\
[e_2,f_2]=f_2\end{matrix}$
 \end{small}
& 
 {\small $\ip_t=\left[ \begin{matrix} t&&&\\ &t&&\\ &&t&\\ &&&t \end{matrix} \right]$, $t>0$}
   & {\small $\left[ \begin{array}{cccc} 0&-a_1&0&0\\ a_1&0&0&0\\ 0&0&0&-a_2\\ 0&0&a_2&0 \end{array} \right]$}
   \\ \hline
{\small $\begin{matrix}  \mathfrak{r}'_{4,\lambda,0},\ \lambda>0:\\ 
[e_1,f_1]=\lambda  f_1, \\ 
 [e_1,f_2]= e_2,\\  
 [e_1,e_2]=- f_2\end{matrix}$}
&  {\small $\ip_t= \left[ \begin{matrix} t&&&\\ &t&&\\ &&t&\\ &&&t \end{matrix} \right]$, $t>0$}
 & {\small $\left[ \begin{array}{cccc} 0&-a_1&0&0\\ a_1&0&0&0\\ 0&0&0&-a_2\\ 0&0&a_2&0 \end{array} \right]$}
 \\
  \hline
{\small $\begin{matrix} \aff(\RR)\times \aff(\RR):\\
  [e_1,f_1]= f_1,\\
[e_2,f_2]=f_2\end{matrix}$}
 & {\small$\ip_{t,s}=\left[ \begin{matrix} t&&&\\ &t&&\\ &&ts&\\ &&&ts \end{matrix} \right]$, $\begin{matrix}
s,t>0\\ s\leq 1
\end{matrix}$} & {\small $\left[ \begin{array}{cccc} 0&-a_1&0&0\\ a_1&0&0&0\\ 0&0&0&-a_2\\ 0&0&a_2&0 \end{array} \right]$}
 \\
\hline
\end{tabular}
\end{spacing}
\caption{{$4$-dimensional metric Lie algebras that admit a parallel tensor not multiple of a complex structure}}
\label{tabla1}
\end{table}
 
\begin{proof} 
Let $(\ggo,\ip)$ be a non-abelian 4-dimensional metric Lie algebra, and let $H$ be any skew-symmetric endomorphism. Fix an orthonormal basis $\{e_1,f_1,e_2,f_2\}$ such that
$H(e_i)=a_if_i$, $H(f_i)=-a_ie_i$ for $i=1,2$.
Assume now that $H$ is parallel, which means that $\nabla_xH=H\nabla_x$ for all $x\in\ggo$, and assume also that $H$ is not a multiple of a complex structure, which implies that $|a_1|\neq |a_2|$. Using Lemma \ref{lemautil},  $\nabla_x$ has the form (with respect to the fixed basis)
\begin{align}\label{nabla in term of alpha and beta}
\nabla_{x}=\left[ \begin{array}{cccc}\ 0&- \alpha(x)&0&0\\ \alpha(x)&0&0&0 \\ 0&0&0&-\beta(x)\\
0&0& \beta(x)&0 \end{array}\right], 
\end{align}
for some linear forms $\alpha$ and $\beta$ on $\ggo$ that are not zero simultaneously. 
The Lie brackets can be expressed in terms of $\alpha$ and $\beta$ using that the Levi-Civita connection is torsion-free:
\begin{align}\label{Lie brackets in terms of alfa and beta}
[e_1,f_1]&=-\alpha(e_1)e_1-\alpha(f_1)f_1,&
[e_1,e_2]&=\beta(e_1)f_2-\alpha(e_2)f_1,  \\
\nonumber [e_1,f_2]&=-\beta(e_1)e_2-\alpha(f_2)f_1, &
[f_1,e_2]&=\beta(f_1)f_2+\alpha(e_2)e_1,  \\
\nonumber [f_1,f_2]&=-\beta(f_1)e_2+\alpha(f_2)e_1, &
[e_2,f_2]&=-\beta(e_2)e_2-\beta(f_2)f_2.
\end{align}
Applying  the Jacobi identity several times we obtain the following relations:
\begin{multicols}{2}
\noindent
\begin{align}
\label{jacobi1}
\alpha(e_1)\beta(e_1)+\alpha(f_1)\beta(f_1)=0
\end{align}
\begin{align} 
\label{jacobi2}
\alpha(e_1)\alpha(e_2)+\alpha(f_2)\beta(f_1)=0
\end{align}
\begin{align}
\label{jacobi3}
-\alpha(e_2)\alpha(f_1)+\alpha(f_2)\beta(e_1)=0
\end{align}
\begin{align}
\label{jacobi4}
\alpha(e_1)\alpha(f_2)-\alpha(e_2)\beta(f_1)=0
\end{align}
\begin{align}
\label{jacobi5}
\alpha(f_1)\alpha(f_2)+\alpha(e_2)\beta(e_1)=0
\end{align}
\begin{align}
\label{jacobi7}
\beta(e_2)\alpha(e_2)+\beta(f_2)\alpha(f_2)=0
\end{align}
\begin{align}
\label{jacobi6}
\beta(e_2)\beta(e_1)+\beta(f_1)\alpha(f_2)=0
\end{align}
\begin{align}
\label{jacobi8}
-\beta(e_1)\beta(f_2)+\beta(f_1)\alpha(e_2)=0
\end{align}
\begin{align}
\label{jacobi9}
\beta(e_2)\beta(f_1)-\beta(e_1)\alpha(f_2)=0
\end{align}
\begin{align}
\label{jacobi10}
\beta(f_2)\beta(f_1)+\beta(e_1)\alpha(e_2)=0.
\end{align}
\end{multicols}


We first rewrite \eqref{jacobi1}-\eqref{jacobi10} as matrix products and  computation of determinants. 
Conditions (\ref{jacobi1}) and (\ref{jacobi7}) can be written as
\begin{align}\label{det U=0, det V=0}
\det \underbrace{\left(\begin{matrix}
\alpha(e_1)&-\beta(f_1)\\ \alpha(f_1)&\beta(e_1)
\end{matrix} \right)}_{=:U}=0, \quad \det \underbrace{\left(\begin{matrix}
\beta(e_2)&-\alpha(f_2)\\ \beta(f_2)&\alpha(e_2)
\end{matrix} \right)}_{=:V}=0.
\end{align}
Conditions (\ref{jacobi2})-(\ref{jacobi5})  can be written as:
\begin{align}\label{AB=0}
\underbrace{\left(
\begin{matrix}
\alpha(e_2)&\alpha(f_2)\\ \alpha(f_2)&-\alpha(e_2)
\end{matrix}\right)}_{=:A}\underbrace{\left(
\begin{matrix}
\alpha(e_1)&\beta(e_1)\\ \beta(f_1)&\alpha(f_1)
\end{matrix}
\right)}_{=:B}=0_{2\times 2},
\end{align}
Similarly, (\ref{jacobi6})-(\ref{jacobi9}) can be written as
\begin{align}\label{CD=0}
\underbrace{\left(\begin{matrix}
\beta(e_1)&\beta(f_1)\\ \beta(f_1)&-\beta(e_1)
\end{matrix}\right)}_{=:C}\underbrace{\left(
\begin{matrix}
\beta(e_2)&\alpha(e_2)\\ \alpha(f_2)&\beta(f_2)
\end{matrix}
\right)}_{=:D}=0_{2\times 2}.
\end{align}
Note that $\det A=-(\alpha(e_2)^2+\alpha(f_2)^2)$ and $\det C=-(\beta(e_1)^2+\beta(f_1)^2)$, from which it follows that $A$ is either zero or invertible, and similar with $C$. 
For the rest of the discussion we consider the inner product $\langle \cdot,\cdot \rangle$ on $\ggo^*$ where $\{e^1,f^1,e^2,f^2\}$, the dual basis of $\{e_1,f_1,e_2,f_2\}$, is orthonormal.

If $C\neq 0$, then it is invertible and (\ref{CD=0}) implies that $D=0$, that is $\alpha|_{\on{Span}\{e_2,f_2\}}=\beta|_{\on{Span}\{e_2,f_2\}}=0$. 
From $\det U=0$ we obtain that $(\alpha(e_1),\alpha(f_1))=\mu(-\beta(f_1),\beta(e_1))$ for some $\mu$. We consider a new orthonormal basis $\{e_1',f_1',e_2,f_2\}$ by setting $e_1':=\frac{\beta(e_1)e_1+\beta(f_1)f_1}{\|\beta\|}$ and $f_1':= \frac{-\beta(f_1)e_1+\beta(e_1)f_1}{\|\beta \|}$. One readily checks that the Lie brackets are given by:  $[e_1',f_1']=-\mu\|\beta\|f_1' $, $[e_1',f_2]=-\|\beta\|e_2 $, $[e_1',e_2]=\|\beta\|f_2 $, and that $H e_1'=a_1f_1'$ and $Hf_1'=-a_1e_1'$. We now analyze the following two cases:
\begin{enumerate}[(i)] 
\item Case $\mu=0$:  We  redefine $\{e_1,f_1,e_2,f_2\}$ as $\{-\frac{e_1'}{\|\beta\|},-\frac{f_1'}{\|\beta\|},\frac{e_2}{\|\beta\|},\frac{f_2}{\|\beta\|}\}$. The Lie brackets are now given by: $[e_1,e_2]=-f_2$ and $[e_1,f_2]=e_2$, and 
$H e_i=a_if_i$, $Hf_i=-a_ie_i$ for $i=1,2$.  Note that this metric Lie algebra is listed in the first row of Table \ref{tabla1} with $t=\frac{1}{\|\beta\|^2}$.
\item Case $\mu\neq 0$: We set $\lambda=|\mu|$ and $\epsilon=-\sign(\mu)$,  redefine  $\{e_1,f_1,e_2,f_2\}$ as $\{\epsilon\frac{e_1'}{\|\beta\|}, \epsilon\frac{f_1'}{\|\beta\|}, \epsilon\frac{e_2}{\|\beta\|},$  $\epsilon\frac{f_2}{\|\beta\|} \}$. The Lie brackets are $[e_1,f_1]=\lambda f_1$, $[e_1,f_2]=e_2$, $[e_1,e_2]=-f_2$, and $H e_i=a_i f_i$, $Hf_i=-a_i e_i$ for $i=1,2$.
Note that this is the metric Lie algebra of the third row of Table \ref{tabla1} with $t=\frac{1}{\|\beta\|^2}$.
\end{enumerate}

If $A\neq 0$, then we interchange $(e_1,f_1,e_2,f_2,\alpha,\beta, a_1,a_2)\leftrightarrow (e_2,f_2,e_1,f_1,\beta,\alpha,a_2,a_1)$ and the conclusion will be the same as in the case $C\neq 0$.

Assume finally that $A=C=0$, that is, $\alpha|_{\on{Span}\{e_2,f_2\}}=\beta|_{\on{Span}\{e_1,f_1\}}=0$. According to (\ref{Lie brackets in terms of alfa and beta}), the Lie brackets are given by $[e_1,f_1]=-\alpha(e_1)e_1-\alpha(f_1)f_1$ and $[e_2,f_2]=-\beta(e_2)e_2-\beta(f_2)f_2$.
We may assume that $\beta\neq 0$, otherwise we interchange $(e_1,f_1,e_2,f_2,\alpha,\beta,a_1,a_2)\leftrightarrow (e_2,f_2,e_1,f_1,\beta,\alpha,a_2,a_1)$ and arrive to this situation. We consider two cases:
\begin{enumerate}[(i)]
    \item Case $\alpha=0$: We consider a new orthonormal basis $\{e_1,f_1,e_2',f_2'\}$
    by setting  $e_2':=\frac{-\beta(f_2)e_2+\beta(e_2)f_2}{\|\beta\|}$ and $f_2':=\frac{-\beta(e_2)e_2-\beta(f_2)f_2}{\|\beta\|}$. The Lie brackets are given by $[e_2',f_2']=\|\beta\| f_2'$, and $He_2'=a_2f_2'$, $Hf_2'=-a_2e_2'$. We redefine $\{ e_1,f_1,e_2,f_2\}$ as  $\{-\frac{e_1}{\|\beta\|},-\frac{f_1}{\|\beta\|},\frac{e_2}{\|\beta\|},\frac{f_2}{\|\beta\|}\}$. The Lie brackets are now given by $[e_2,f_2]=f_2$.  This is the metric Lie algebra listed in the second row of Table \ref{tabla1} with  $t=\frac{1}{\|\beta\|^2}$.
    \item Case $\alpha\neq 0$: We consider a new orthonormal basis $\{e_1',f_1',e_2',f_2'\}$ by setting $e_1'=\frac{\alpha(f_1)e_1-\alpha(e_1)f_1}{\|\alpha\|}$, $f_1'=\frac{\alpha(e_1)e_1+\alpha(f_1)f_1}{\|\alpha\|}$, $e_2'=\frac{\beta(f_2)e_2-\beta(e_2)f_2}{\|\beta\|}$,  $f_2'=\frac{-\beta(e_2)e_2-\beta(f_2)f_2}{\|\beta\|}$.
 The Lie brackets are given by: $[e_1',f_1']=\|\alpha\| f_1'$, $[e_2',f_2']=\|\beta\| f_2'$, and $H e_i'=a_if_i'$ and $Hf_i'=-a_ie_i'$, $i=1,2$. We may assume that $\|\alpha\|\leq \|\beta\|$, otherwise we interchange $(e_1,f_1,e_2,f_2,\alpha,\beta)\leftrightarrow (e_2,f_2,e_1,f_1,\beta,\alpha)$. 
We now redefine $\{e_1,f_1,e_2,f_2\}$ as $\{\frac{e_1'}{\|\alpha\|},\frac{f_1'}{\|\alpha\|},\frac{e_2'}{\|\beta\|},\frac{f_2'}{\|\beta\|}\}$. The Lie brackets are now given by: $[e_1,f_1]=f_1$, $[e_2,f_2]=f_2$. We obtain the metric Lie algebra listed in the fourth row of Table \ref{tabla1}, with $t=\frac{1}{\|\alpha\|^2}$ and $s=\frac{\|\alpha\|^2}{\|\beta\|^2}$. Note that $He_i=a_if_i$ and $Hf_i=-a_ie_i$.
\end{enumerate}
\color{black}

We have completed all the rows of Table \ref{tabla1}. 

Finally, fix a metric Lie algebra of Table \ref{tabla1}. Suppose $T$ is any a skew-symmetric endomorphism given as in third column of Table \ref{tabla1} for the respective metric Lie algebra. Then computing $\nabla_{e_i},\nabla_{f_i}$ 
  it is straightforward to check that $T$ commutes with $\nabla_{e_i},\nabla_{f_i}$. Thus $T$ is parallel. 


\end{proof}
\begin{rem}\label{a1a2}
When  $|a_1|= |a_2|$ in each metric Lie algebra of Table \ref{tabla1}, $H$ is a multiple of a complex structure. This case was studied in \cite{Ov} as we said previously. Furthermore, the metric Lie algebras of Table \ref{tabla1} are the same metric Lie algebras of the first four rows of Table \ref{tabla11}.
\end{rem}

On a fixed metric Lie algebra $(\ggo,\ip)$,  the set of parallel skew-symmetric endomorphisms form a vector space. As a consequence of Theorem \ref{dim4} and Remark \ref{a1a2} we obtain the following corollary:
\begin{cor}
The vector space of parallel skew-symmetric endomorphisms on each metric Lie algebra of Table \ref{tabla1} has dimensi\'on $2$. 

\end{cor}

From Theorem \ref{dim4} and comparing Table \ref{tabla1} with Table \ref{tabla11}, we have:
\begin{cor}
Let $\ggo$ be one of the Lie algebras  $\mathfrak{d}_{4,\frac{1}{2}},\mathfrak{d}_{4,2}, \mathfrak{d}_{4,\frac{\delta}{2}}'$ with $\delta>0$, and let $\ip$ be one of the metrics of Table \ref{tabla11}. Then the only parallel skew-symmetric endomorphisms on $(\ggo,\ip)$ are multiple of complex structures. Thus the vector space of parallel skew-symmetric endomorphisms on these metric Lie algebras is $1$-dimensional.
\end{cor}
We present the information of the above corollary in  Table \ref{tabla de deltas}. We choose a presentation so that the outputs of Table \ref{tabla1} and Table \ref{tabla de deltas} look similar. Namely, we fix the structure coefficients of the Lie algebras and vary the metric, and we also write all the non-zero parallel tensors, not just the complex structures. \color{black} 

\begin{table}[h!]
\centering
\begin{spacing}{1.5}
\begin{tabular}{l l l}
\hline
\hline
Lie algebra &  Metric   & Parallel tensor  \\
\hline \hline
 $\begin{matrix} \mathfrak{d}_{4,\frac{1}{2}}:\\
 [e_1,e_2]=e_3,  \\ [e_4,e_1]=\frac{1}{2}e_1, \\ [e_4,e_2] =\frac{1}{2}e_2, \\ [e_4,e_3]=e_3 \end{matrix}   $
  & {\small $\ip_t=\left[\begin{matrix}
  t&0&0&0\\ 0&t&0&0\\ 0&0&t&0\\ 0&0&0&t
\end{matrix}\right]$, $t>0$   }     & {\small $
c\left[\begin{matrix}
0&-1&0&0\\ 1&0&0&0\\ 0&0&0&1\\ 0&0&-1&0
\end{matrix}\right]$}, {\small  $c\in \mathbb{R}^*$ }\\
 \hline
  $ \begin{matrix} \mathfrak{d}_{4,2}:\\
  [e_1,e_2]=e_3, \\ [e_4,e_1]=e_1\\ [e_4,e_2] =-\frac{1}{2}e_2, \\   [e_4,e_3]=\frac{1}{2}e_3,\end{matrix} $ & {\small $\ip_t=\left[\begin{matrix}
  t&0&0&0\\ 0&t&0&0\\ 0&0&t&0\\ 0&0&0&t
\end{matrix}\right]$, $t>0$   }  &{\small $
c\left[\begin{matrix}
0&0&0&-1\\ 0&0&-1&0\\ 0&1&0&0\\ 1&0&0&0
\end{matrix}\right]$},
{\small  $c\in \mathbb{R}^*$ } 
 \\ \hline
  $   \begin{matrix} {\mathfrak{d}'_{4,\frac{\delta}{2}}},\quad \delta >0: \\ [e_1,e_2]=e_3,  
 \\ [e_4,e_1]=\frac{1}{2}e_1-\frac{1}{\delta}e_2, \\ [e_4,e_2]=\frac{1}{\delta}e_1+\frac{1}{2}e_2, \\ [e_4,e_3]=e_3 \end{matrix} $ &{\small $\ip_t=\left[\begin{matrix}
  t&0&0&0\\ 0&t&0&0\\ 0&0&t&0\\ 0&0&0&t
\end{matrix}\right]$, $t>0$   }  & 
 {\small $
c\left[\begin{matrix}
0&-1&0&0\\ 1&0&0&0\\ 0&0&0&1\\ 0&0&-1&0
\end{matrix}\right]$}, {\small  $c\in \mathbb{R}^*$ }
 \\
   
     \hline
   \end{tabular}
   \caption{Parallel tensors on $(\dd_{4,\frac{1}{2}},\ip_t)$, $(\dd_{4,2},\ip_t)$, $(\dd'_{4,\frac{\delta}{2}},\ip_t)$}
   \label{tabla de deltas}
    \end{spacing}
   \end{table}


\medskip

We now address the problem of distinguishing the metric Lie algebras and the parallel tensors presented in Table \ref{tabla1} and Table \ref{tabla de deltas}.
Two metric Lie algebras $(\ggo,\ip)$ and $(\ggo',\ip')$ are said to be equivalent if there exists an isometric Lie algebra isomorphism $\ggo\cong\ggo'$. In notation, $(\ggo,\ip)\sim (\ggo',\ip')$. Proposition \ref{equivalencias1} below shows that the metric Lie algebras of Table \ref{tabla1} and Table \ref{tabla de deltas} are all pairwise non-equivalent.
In addition, for a fixed metric Lie algebra $(\ggo,\ip)$, it describes the moduli space of the parallel endomorphisms according to Definition \ref{equivalentes}.
\begin{rem}\label{invariantes clasicos}
During the proof of the following proposition we will use that if $\phi:(\ggo,\ip)\to (\ggo,\ip')$ is an isometric Lie algebra automorphism, then the Ricci operators satisfy the relation $\on{Ric}_{\ip'}\phi=\phi\on{Ric}_{\ip}$; in particular, 
\begin{enumerate}
    \item[(a)] $\on{Ric}_{\ip}$ and $\on{Ric}_{\ip'}$ have the same characteristic polynomial;
    \item[(b)] if in addition $\ip=\ip'$, then $\phi$ preserves the eigenspaces of $\on{Ric}_{\ip}$.
\end{enumerate}
We will also use that if $H_1$ and $H_2$ are equivalent parallel endomorphisms on $(\ggo,\ip)$, so that there exists an isometric Lie algebra automorphism $\phi:\ggo\to\ggo$ such that $H_2\phi=\phi H_1$, then
\begin{enumerate}
    \item[(c)] $H_1$ and $H_2$ have the same characteristic polynomial.
    \item[(d)] $H_1^2$ and $H_2^2$ have the same eigenvalues and $\phi$ preserves the corresponding eigenspaces.
    \item[(e)] Moreover, if $H_1$ preserves $[\ggo,\ggo]$ then the same holds for $H_2$, and the restrictions of $H_1$ and $H_2$ to $[\ggo,\ggo]$ have the same characteristic polynomial. An analogous conclusion holds if $H_1$ preserves $[\ggo,[\ggo,\ggo]]$ or $[[\ggo,\ggo],[\ggo,\ggo]]$, or $\mathfrak{z}(\ggo)$, etc.
\end{enumerate}
 
\end{rem}
\begin{prop}
\label{equivalencias1}
The metric Lie algebras of Table \ref{tabla1} and Table \ref{tabla de deltas} are pairwise non-equivalent.  Given one of these metric Lie algebras, any  parallel skew-symmetric endomorphism  is equivalent to exactly one of the endomorphisms given in Table \ref{tabla113}.

\end{prop}
\begin{table}[h!]
\centering
\begin{spacing}{1.5}
\begin{tabular}{l l}
\hline
\hline
 Metric Lie algebra  & Parallel endomorphism   \\
\hline \hline 
  $( \RR \times \mathfrak{e}(2),\ip_t)$
  &   {\tiny $\left[ \begin{array}{cccc} 0&-a_1&0&0\\ a_1&0&0&0\\ 0&0&0&-a_2\\ 0&0&a_2&0 \end{array} \right]$, $a_1,a_2\geq 0$} \\
 \hline 
$
(\RR^2\times \aff(\RR),\ip_t)$
   & {\scriptsize  $\left[ \begin{array}{cccc} 0&-a_1&0&0\\ a_1&0&0&0\\ 0&0&0&-a_2\\ 0&0&a_2&0 \end{array} \right]$, $a_1,a_2\geq 0$}  \\ \hline
$ (\mathfrak{r}'_{4,\lambda,0},\ip_t)$, $ \lambda>0$ & 
{\tiny $\left[ \begin{array}{cccc} 0&-a_1&0&0\\ a_1&0&0&0\\ 0&0&0&-a_2\\ 0&0&a_2&0 \end{array} \right]$, $a_1\geq 0$} 
 \\
  \hline
 $(\aff(\RR)\times \aff(\RR),\ip_{t,s})$, $s,t>0, \, s\leq 1$
& {\tiny $\left[ \begin{array}{cccc} 0&-a_1&0&0\\ a_1&0&0&0\\ 0&0&0&-a_2\\ 0&0&a_2&0 \end{array} \right]$, $\begin{matrix} a_1,a_2\geq 0\\ \mbox{if } s<1,\ \mbox{or}\\ a_1\geq a_2\geq 0\\ \mbox{if } s=1 \end{matrix}$} \\
\hline
$(\mathfrak{d}_{4,\frac{1}{2}},\ip_t)$ & \begin{tiny}
  $
c\left[\begin{matrix}
0&-1&0&0\\ 1&0&0&0\\ 0&0&0&1\\ 0&0&-1&0
\end{matrix}\right]$, $c>0$
\end{tiny} \\ \hline
$(\mathfrak{d}_{4,2},\ip_t)$  & \begin{tiny} $
c\left[\begin{matrix}
0&0&0&-1\\ 0&0&-1&0\\ 0&1&0&0\\ 1&0&0&0
\end{matrix}\right]$, $c>0$
\end{tiny} \\ \hline 
$(\mathfrak{d}'_{4,\frac{\delta}{2}},\ip_t), \delta>0$ &\begin{tiny}
$
c\left[\begin{matrix}
0&-1&0&0\\ 1&0&0&0\\ 0&0&0&1\\ 0&0&-1&0
\end{matrix}\right]$, $c\in \mathbb{R}^*$
\end{tiny}
\\
\hline
\end{tabular}
\end{spacing}
\caption{{Non equivalent parallel tensors on non equivalent $4$-dimensional metric Lie algebras}}
\label{tabla113}
\end{table}
\begin{proof}
The Lie algebras of the first columns of Table \ref{tabla1} and Table \ref{tabla de deltas} are all pairwise non isomorphic.
We now proceed case by case.

\smallskip 
$\ggo=\RR \times \mathfrak{e}(2)$: 
    Let $\phi:(\ggo,\ip_t)\to (\ggo,\ip_{t'})$ be an isometric Lie algebra isomorphism. Since  $[\ggo,\ggo]=\Span\{e_2,f_2\}$ and  $\mathfrak{z}(\ggo)=\Span\{f_1\}$, $\phi$ has the form
\[ \phi=\left[ \begin{matrix} x_1&0&0&0 \\ 0&x_4&0&0 \\ 0&0&x_9&-x_{10} \\ 0&0&x_{10}&x_9 \end{matrix} \right],  \]
with $x_1^2t'=x_4^2t'=t$ and $(x_9^2+x_{10}^2)t'=t$. One readily checks that $\phi([e_1,e_2])=[\phi(e_1),\phi(e_2)]$, so that $x_1=1$ and hence $t=t'$. 

Now fix $t$ and suppose that $ H_{a_1,a_2}$ is a skew-symmetric endomorphism as in Table \ref{tabla1}. 
It is possibly to multiply the basis elements by $\pm 1$'s so that the Lie brackets in the new basis are the same as before and the endomorphism has
 $a_1,a_2\geq 0$. For instance, if $a_1\geq 0$ and $a_2<0$, then we do the change $(e_1,f_1,e_2,f_2)\mapsto (-e_1,-f_1,-e_2,f_2)$. 
It is only left to show that if $ H_{a_1,a_2}$ and $ H_{a_1',a_2'}$ are equivalent, with $a_i,a_i'\geq 0$, then $(a_1,a_2)=(a_1',a_2')$. By Remark \ref{invariantes clasicos}, (c), we have $(X^2+a_1^2)(X^2+a_2^2)=(X^2+a_1'^2)(X^2+a_2'^2)$; in particular $a_1^2+a_2^2=a_1'^2+a_2'^2$. 
Moreover, since both $ H_{a_1,a_2}$ and $H_{a_1',a_2'}$ preserve $[\ggo,\ggo]=\Span\{e_2,f_2\}$, Remark \ref{invariantes clasicos}, (e), tells us that $X^2+a_2^2=X^2+a_2'^2$. It follows that $a_2=a_2'$ and hence $a_1=a_1'$.

\smallskip

$\ggo=\RR^2\times \aff(\RR)$: One easily checks that the Ricci operator of $(\ggo,\ip_t)$ with respect to the orthonormal basis $\{\frac{e_1}{\sqrt{t}},\frac{f_1}{\sqrt{t}}, \frac{e_2}{\sqrt{t}},\frac{f_2}{\sqrt{t}}\}$ is given by:
\[\Ric_t=\left[  \begin{matrix}0&0&0&0\\0&0&0&0\\0&0&-\frac{1}{t}&0\\0&0&0&
-\frac{1}{t}\end{matrix}\right].\]
It follows from Remark \ref{invariantes clasicos}, (a), that $(\ggo,\ip_t)\sim(\ggo,\ip_{t'})$ if and only if $t=t'$.

 We now fix $t>0$ and suppose that $ H_{a_1,a_2}$ is a skew-symmetric endomorphism as in Table \ref{tabla1}. By multiplying $f_1$ and/or $f_2$ by $\pm 1$, which does not affect the Lie brackets, we can take $a_1,a_2\geq 0$.  Now suppose that $H_{a_1,a_2}$ and $H_{a_1',a_2'}$ are equivalent, with $a_i,a_i'\geq 0$. By Remark \ref{invariantes clasicos}, (c), $(X^2+a_1^2)(X_2^2+a_2^2)=(X^2+a_1'^2)(X^2+a_2'^2)$; in particular $a_1^2+a_2^2=a_1'^2+a_2'^2$. Moreover, since both $H_{a_1,a_2}$ and $H_{a_1',a_2'}$ preserve  $\mathfrak{z}(\ggo)=\Span\{ e_1,f_1\}$, Remark \ref{invariantes clasicos}, (e), tells us that $X^2+a_1^2=X^2+a_1'^2$. It follows that $a_1=a_1'$ and hence $a_2=a_2'$.

$\ggo=\mathfrak{r}'_{4,\lambda,0}$, $\lambda>0$: The Ricci operator of $(\ggo,\ip_t)$ with respect to the orthonormal basis $\{ \frac{e_1}{\sqrt{t}},\frac{f_1}{\sqrt{t}},\frac{e_2}{\sqrt{t}},\frac{f_2}{\sqrt{t}}\}$, is given by:
\[ \Ric_t=\left[ \begin{matrix} -\frac{\lambda^2}{t}&0&0&0\\0&0&0&0\\0&0&0&0\\0&0&0&0 \end{matrix}\right]. \] 
By Remark \ref{invariantes clasicos}, (a), we obtain that $(\ggo,\ip_t)\sim (\ggo,\ip_{t'})$ if and only if $t=t'$.

We now fix $t>0$. Let $ H_{a_1,a_2}$ be a  skew-symmetric endomorphism as in Table \ref{tabla1}. After changing $f_1$ by $-f_1$ if necessary, which does not affect the Lie brackets, we can take $a_1\geq 0$. We now assume that $H_{a_1,a_2}$ and $H_{a_1',a_2'}$ are equivalent, with $a_1,a_1'\geq 0$ and $a_2,a_2'\in\mathbb{R}$.  By Remark \ref{invariantes clasicos}, (c),  $(X^2+a_1^2)(X^2+a_2^2)=(X^2+a_1'^2)(X^2+a_2'^2)$, in particular $a_1^2+a_2^2=a_1'^2+a_2'^2$.  Now let $\phi:\ggo\to\ggo$ be an isometric isomorphism such that $\phi H_{a_1,a_2}=H_{a_1',a_2'}\phi$. 
Then $\phi$ preserves $[[\ggo,\ggo],[\ggo,\ggo]]=\Span\{e_2,f_2\}$, hence $\phi(e_2)=ue_2+vf_2$ and $\phi(f_2)=-ve_2+uf_2$ with $u^2+v^2=1$. 
Now we have $\phi H_{a_1,a_2}(e_2)=a_2\phi(f_2)=-a_2 ve_2+a_2uf_2$ and $H_{a_1',a_2'}\phi(e_2)=H_{a_1',a_2'}(ue_2+vf_2)=a_2'uf_2-a_2've_2$. From this we see that $a_2=a_2'$, and hence  $a_1=a_1'$.

\medskip

$\ggo=\aff(\RR)\times \aff(\RR)$:  The Ricci operator with respect to the orthonormal basis $\{\frac{e_1}{\sqrt{t}},\frac{f_1}{\sqrt{t}},\frac{e_2}{\sqrt{ts}},\frac{f_2}{\sqrt{ts}}\}$ is given by:
 \[\Ric_{s,t}=\left[  \begin{matrix} -\frac{1}{t}&0&0&0\\0&-\frac{1}{t}&0&0\\0&0&-\frac{1}{st}&0
\\0&0&0&
-\frac{1}{st}\end{matrix}\right].\]
It follows from Remark \ref{invariantes clasicos}, (a), that  $(\ggo,\ip_{s,t})\sim (\ggo,\ip_{s',t'})$ if and only if $t=t'$ $s=s'$ (here we are using the constraint $0<s,s'\leq 1$).

We now fix $t>0$ and $0< s\leq 1$. Let $H_{a_1,a_2}$ be a skew-symmetric endomorphism as in Table \ref{tabla1}. We can multiply $f_1$ or $f_2$ by $-1$ if necessary, without changing the Lie brackets, and take $a_1, a_2\geq 0$. If $s=1$ we can interchange $\{e_1,f_1\}$ and $\{e_2,f_2\}$ if necessary and assume that $a_1\geq a_2$. 

Now suppose that there exists an isometric isomorphism $\phi:\ggo\to\ggo$ such that
 $H_{a_1,a_2}\phi=H_{a_1',a_2'}\phi$, where $a_i,a_i'\geq 0$ and also $a_1\geq a_2$ and $a_1'\geq a_2'$ in the case $s=1$.  We have to show that $a_1=a_1'$ and $a_2=a_2'$.
 We consider two cases:
\begin{enumerate}[(i)]
    \item $s<1$: Since $\Span\{e_1,f_1\}$ and $\Span\{e_2,f_2\}$ are the eigenspaces of the Ricci operator, they are preserved by $\phi$ according to Remark \ref{invariantes clasicos}, (b). This clearly implies that $a_1=a_1'$ and $a_2=a_2'$.
    
\item  $s=1$: In this case we have to assume that $a_1\geq a_2\geq 0$, $a_1'\geq a_2'\geq 0$. By Remark \ref{invariantes clasicos}, (c), $(X^2+a_1^2)(X^2+a_2^2)=(X^2+a_1'^2)(X^2+a_2'^2)$. From this we deduce that $a_1=a_1'$ and $a_2=a_2'$.

\end{enumerate}
\medskip

$\ggo= \mathfrak{d}_{4,\frac{1}{2}}$: In the orthonormal basis 
$\{\frac{e_1}{\sqrt{t}},\frac{e_2}{\sqrt{t}},\frac{e_3}{\sqrt{t}},\frac{e_4}{\sqrt{t}}\}$, the Ricci operator is    
\[ \Ric_t=\left[ \begin{matrix} -\frac{3}{2t}&0&0&0\\ 0&-\frac{3}{2t}&0&0\\ 0&0&-\frac{3}{2t}&0\\0&0&0&-\frac{3}{2t} \end{matrix}\right]. \]

It follows from Remark \ref{invariantes clasicos}, (a), that $(\ggo,\ip_t)\sim (\ggo,\ip_{t'})$ if and only if $t=t'$. 

We now fix $t$. Given a  skew-symmetric endomorphism $H_c$ as in Table \ref{tabla de deltas}, we can change $(e_1,e_3)$ by $(-e_1,-e_3)$ if necessary, without changing the Lie brackets, and take $c>0$. Using Remark \ref{invariantes clasicos} (d), we now easily see that if $H_c$ and $H_{c'}$ are equivalent, with $c,c'>0$, then $c=c'$.
\medskip

 $\ggo=\mathfrak{d}_{4,2}$: The Ricci operator in the orthonormal basis  $\{ \frac{e_1}{\sqrt{t}},\frac{e_2}{\sqrt{t}},\frac{e_3}{\sqrt{t}},\frac{e_4}{\sqrt{t}}\}$ is
\[ \Ric_t=\left[ \begin{matrix} -\frac{3}{2t}&0&0&0\\ 0&0&0&0\\ 0&0&0&0\\0&0&0&-\frac{3}{2t} \end{matrix}\right]. \]
 It follows from Remark \ref{invariantes clasicos}, (a), that if $(\ggo,\ip_t)\sim(\ggo,\ip_{t'})$, then $t=t'$. 
 
  We now fix $t$ and let $H_c$ be an endomorphism as in Table \ref{tabla de deltas}. By changing $(e_1,e_2)$ by $(-e_1,-e_2)$ if necessary, we can take $c>0$. Finally, if $H_{c}$ and $H_{c'}$ are equivalent, then by Remark \ref{invariantes clasicos}, (d), we get $c=c'$.

$\ggo=\dd'_{4,\frac{\delta}{2}}$: We consider the orthonormal basis $\{\frac{e_1}{\sqrt{t}},\frac{e_2}{\sqrt{t}},\frac{e_3}{\sqrt{t}}, \frac{e_4}{\sqrt{t}}\}$. The Ricci operator  is given by
\[ \Ric=\left[ \begin{matrix} -\frac{3}{2t}&0&0&0\\ 0&-\frac{3}{2t}&0&0\\ 0&0&-\frac{3}{2t}&0\\0&0&0&-\frac{3}{2t} \end{matrix}\right]. \]
It follows from Remark \ref{invariantes clasicos}, (a), that if $(\ggo,\ip_t)\sim (\ggo,\ip_{t'})$, then $t=t'$.  

Fix $t>0$. Given two skew-symmetric endomorphisms $H_c$ and $H_{c'}$ as in Table \ref{tabla de deltas}, with $c,c'\in \mathbb{R}^*$, by Remark \ref{invariantes clasicos}, (d), we have $|c|=|c'|$.
 We claim that $H_c$ and $H_{-c}$ are  non-equivalent. To show this can assume that $c=1$. Suppose that there is an isometric isomorphism $\phi$ such that $H_1\phi=\phi H_{-1}$. Observe that $\ggo':=[\ggo,\ggo]=\Span\{ e_1,e_2,e_3\}$, and that $\ggo'':=[\ggo',\ggo']=\Span\{e_3\}$. Since $\phi$ preserves $\ggo'$ and is an isometry, it also preserves the orthogonal complement, that is, $\phi(e_4)=ke_4$ with $k=\pm 1$. Since it also preserves $\ggo''$, $\phi(e_3)=le_3$ with $l=\pm 1$. Now $le_3=\phi(e_3)=\phi[e_4,e_3]=[\phi(e_4),\phi(e_3)]=kle_3$, from which $k=1$. So $\phi$ can be written in the basis $\{e_1,e_2,e_3,e_4\}$ as
\begin{align*}
\phi=\left[\begin{matrix}
x_1& -x_2&0&0\\
x_2&x_{1}&0&0\\
0&0&l&0\\
0&0&0&1
\end{matrix}\right].
\end{align*}
If we impose the condition that $H_1\phi=\phi H_{-1}$ we get that $x_{1}=x_{2}=0$, which contradicts the fact that $\phi$ is inversible. Hence $H_{1}$ and $H_{-1}$ are not equivalent.

\end{proof}

\section{The de Rham decomposition of the associated simply connected Riemannian Lie groups}\label{The De Rham decomposition}

Given a Riemannian manifold $(M,g)$, the holonomy group at a point $p\in M$, denoted by $\Hol_p(M,g)$, is the group formed by parallel translations $P_{\gamma}$ around loops $\gamma:[a,b]\to M$ at $p$, with the usual product of endomorphisms. 
This is a Lie subgroup of the orthogonal group $O(T_pM)$, and if $M$ is orientable, $\Hol_p(M,g)\subseteq SO(T_pM)$. The restricted holonomy group $\on{Hol}_p^0(M,g)$ is the connected normal subgroup that results from using only contractible loops. This is a closed subgroup of $O(T_pM)$ and hence its action on $T_pM$ is completely reducible. 
If $M$ is connected, then $\Hol_p(M,g)$ (resp., $\Hol_p^0(M,g)$) is conjugate to $\Hol_q(M,g)$ (resp., $\Hol_q^0(M,g)$) for all $p,q\in M$; therefore the holonomy group and the restricted holonomy group do not depend on the base point, and they are denoted  by $\Hol(M,g)$ and $\Hol^0(M,g)$ respectively. We say that $(M,g)$ is irreducible if the action of the restricted holonomy group on $T_pM$ is irreducible.
If $M$ is simply connected, then one has $\Hol(M,g)=\Hol^0(M,g)$. If in addition $M$ is complete, then according to the de Rham theorem (\cite[Theorem 10.3.1]{Petersen}), a decomposition of $T_pM$ into a irreducible subspaces under the action of $\Hol(M,g)$ corresponds to a decomposition of $M$ as a product of irreducible Riemannian manifolds. We refer to this as the de Rham decomposition of $(M,g)$.

Given a metric Lie algebra $(\ggo, \ip)$, we denote by $(G,g)$  the associated simply connected Lie group $G$ with the corresponding  left invariant metric $g$. The pair $(G,g)$ is a complete  Riemannian manifold.

We will see next that if $(\ggo,\ip)$ belongs to Table \ref{tabla de deltas}, then $(G,g)$ is irreducible as a Riemannian manifold. One way to show this is by identifying $(G,g)$ with an already known irreducible Riemannian manifold. See Remark \ref{clase de iso de las delta}. Another possibility is by showing that $\ggo$ has no proper subspaces invariant under the action of $\Hol(G,g)$.
 Since $\Hol(G,g)=\Hol^0(G,g)$ is connected, this is equivalent to showing that $\ggo$ has no proper subspaces invariant by the elements of $\mathfrak{hol}(G,g)$, the Lie algebra of $\Hol(G,g)$.
 In order to know what $\mathfrak{hol}(G,g)$ is, we use the Ambrose-Singer Holonomy Theorem, which states that this algebra is the subalgebra of $\mathfrak{so}(\mathfrak{g},\ip)$ spanned by the curvature operators and all their covariant derivatives.
  We begin by following this approach.

\begin{prop} 
\label{irreducible}
Let $(\ggo,\la \cdot , \cdot \ra)$ be one of the metric Lie algebras of Table \ref{tabla de deltas} and let $(G,g)$ be the associated simply connected  Lie group with the corresponding left invariant metric. Then $(G,g)$ is irreducible as Riemannian manifold.
\end{prop}

\begin{proof}
We do the proof only for $(\ggo,\ip)=(\mathfrak{d}_{4,2},\ip_{t}) $, the other cases being similar 
. Since all the metrics $\ip_t$  are homothetic, we can assume $t=1$ and we simply denote $\ip=\ip_1$. We can identify $(\ggo,\ip)$ with the Euclidean space $\mathbb{R}^4$ so that $\{e_1,e_2,e_3,e_4\}$ is identified with the canonical basis.

As remarked previously, we first have to find $\mathfrak{hol}(G,g)$, which is the Lie subalgebra of $\mathfrak{so}(\ggo,\ip)$ spanned as vector space by the curvature tensors and their covariant derivatives. Then we have to show that $\ggo$ has no proper subrepresentations of $\mathfrak{hol}(G,g)$. 

First, one can check that $\nabla_{e_4}=0$ and that
\begin{scriptsize}
\begin{align*}
\nabla_{e_1}&= \left[\begin{array}{cccc} 0 &  0& 0 & -1\\0 & 0 & -\frac{1}{2} & 0\\0 & \frac{1}{2} & 0 &0 \\1 & 0 & 0 & 0 \end{array} \right], &
\nabla_{e_2}&= \left[\begin{array}{cccc} 0 & 0 & \frac{1}{2} & 0\\0 & 0 & 0 & \frac{1}{2}\\ -\frac{1}{2} & 0 & 0 & 0\\0 & -\frac{1}{2} & 0 & 0 \end{array} \right],  & \nabla_{e_3}&= \left[\begin{array}{cccc} 0 & \frac{1}{2} & 0 & 0\\
-\frac{1}{2} & 0 & 0 & 0\\ 
0 & 0 & 0 & -\frac{1}{2}\\
 0 & 0 & \frac{1}{2} & 0  \end{array} \right].
 \end{align*}
\end{scriptsize}
Recall that $R(x,y)=\nabla_x\nabla_y-\nabla_y\nabla_x
-\nabla_{[x,y]}$ is the curvature tensor for all $x,y \in \ggo$, we now compute 
\begin{align*}
R(e_1,e_2)&=-\frac{1}{2}\nabla_{e_3}, &
R(e_1,e_3)&=-\frac{1}{2}\nabla_{e_2},\\
R(e_1,e_4)&=\nabla_{e_1}, &
R(e_2,e_3)&=\left[\begin{matrix} 0&0&0&-\frac{1}{2} \\ 0&0&\frac{1}{2}&0 \\ 0&-\frac{1}{2}&0&0\\ \frac{1}{2}&0&0&0 \end{matrix} \right] , \\
R(e_2,e_4)&=-\frac{1}{2} \nabla_{e_2}, & R(e_3,e_4)&=\frac{1}{2}\nabla_{e_3}.
\end{align*}

In order to facilitate the computation of the covariant derivatives of these tensors,
we will make use of the natural identification $\mathfrak{so}(\mathfrak{g},\ip)\cong \ggo^*\wedge\ggo^*$. The tensor $R(e_i,e_j)$ is identified with the 2-form $R^{ij}$ defined by, 
\[R^{ij}(x,y)=\langle R(e_i,e_j)x,y\rangle,\quad x,y\in\ggo.\]
It follows that
\begin{align*}
R^{12}&=\frac{1}{4}\left(e^{12}-e^{34}\right), & R^{13}&=\frac{1}{4}(e^{13}+e^{24}), & R^{14}&=e^{14}+\frac{1}{2}e^{23},\\
R^{23}&=\frac{1}{2}\left(e^{14}-e^{23}\right), &
R^{24}&=\frac{1}{4}\left(e^{13}+e^{24}\right), & 
R^{34}&=-\frac{1}{4}(e^{12}-e^{34}),
\end{align*}
where $\{e^1,e^2,e^3,e^4\}$ is the dual basis of $\{e_1,e_2,e_3,e_4\}$ and $e^{ij}$ denotes  $e^i\wedge e^j$. 

Using that $\nabla_x \left(\theta_1 \wedge \theta_2\right)=\nabla_x \theta_1 \wedge \theta_2 +\theta_1 \wedge \nabla_x \theta_2$ for $\theta_1,\theta_2$ 1-forms, and that $\nabla_x y^*(z)= \langle \nabla_x y,z \rangle$ where $y^*$ is the linear functional such that $y^*(x)=\langle x,y\rangle$,  we get the following relations
\begin{align*}
\nabla_{e_1} R^{12}&=-\frac{1}{8}(e^{13}+e^{24}), &\nabla_{e_2} R^{12}&=\frac{1}{4}(e^{23}-e^{14}), \\ \nabla_{e_3} R^{12}&=0,&
\nabla_{e_1} R^{13}&=-\frac{1}{8}(e^{12}-e^{34}), \\ \nabla_{e_2} R^{13}&=0, & \nabla_{e_2} R^{13}&=\frac{1}{4}(e^{14}-e^{23}),\\
\nabla_{e_1} R^{14}&=0, & \nabla_{e_2} R^{14}&=\frac{1}{4}(e^{12}-e^{34}), \\ \nabla_{e_3} R^{14}&=-\frac{1}{4}(e^{24}+e^{13}),&
\nabla_{e_1} R^{23}&=0, \\ \nabla_{e_2} R^{23}&=\frac{1}{2}(e^{12}-e^{34}), & \nabla_{e_3} R^{23}&=-\frac{1}{2}(e^{13}+e^{24}). 
\end{align*}
We see that the linear space spanned by  $R^{ij}$'s and  $\nabla_{e_k} R^{ij}$'s for $i,j,k=1,2,3,4$  is $\Span \{ e^{14}+\frac{1}{2}e^{23},e^{13}+e^{24},e^{12}-e^{34}, e^{14}-e^{23} \}=\Span\{e^{14},e^{23}, e^{13}+e^{24},e^{12}-e^{34}\}$. We conclude that $\mathfrak{hol}(G,g)$ is
\begin{small}
$$
\Span\left\{ \underbrace{\begin{bmatrix}
0&0&0&-1\\ 0&0&0&0\\0&0&0&0\\1&0&0&0
\end{bmatrix}}_{=:A},
\underbrace{\begin{bmatrix}
0&0&0&0\\ 0&0&-1&0\\0&1&0&0\\0&0&0&0
\end{bmatrix}}_{=:B},
\underbrace{\begin{bmatrix}
0&0&-1&0\\ 0&0&0&-1\\1&0&0&0\\0&1&0&0
\end{bmatrix}}_{=:C},
\underbrace{\begin{bmatrix}
0&-1&0&0\\ 1&0&0&0\\0&0&0&1\\0&0&-1&0
\end{bmatrix}}_{=:D}
\right\}
.$$	
\end{small}

We now have to show there are no proper subspaces of $\ggo=\mathbb{R}^{4\times 1}$ invariant by $A$, $B$, $C$ and $D$.
Suppose $W$ is a non-zero invariant subspace of $\ggo$, and let $x=(a_1,a_2,a_3,a_4)^T$ be a non-zero element of $W$. After multiplying by $C$, if necessary, we can assume that $(a_1,a_4)\neq (0,0)$. Then $Ax=(-a_4,0,0,a_1)^T\in W$, $CAx=(0,-a_1,-a_4,0)^T\in W$ and $DCAx=(a_1,0,0,a_4)\in W$. Note that $Ax$ and $DCAx$ are linearly independent, so we obtain that $e_1,e_4\in W$. Similarly, we get that $e_2,e_3\in W$ and therefore $W=\ggo$. The proof is now complete. 
\end{proof}
\begin{rem}\label{clase de iso de las delta} 
The Proposition \ref{irreducible} can also be shown identifying $(G,g)$ with known Riemannian manifolds.
One can easily see that $(\mathfrak{d}_{4,\frac{1}{2}},\la \cdot, \cdot \ra_{t})$ is equivalent 
up to scaling to the metric Lie algebra with parameter zero in the second family listed in the main theorem of \cite{Jensen}. 
 Analogously, $(\mathfrak{d}'_{4,\frac{\delta}{2}},\la \cdot, \cdot \ra_{t})$ is equivalent up to scaling to the metric Lie algebra with parameter $\frac{2}{\delta}$ of the same family. In both cases $(G,g)$ is isometric up to scaling to the Hermitian hyperbolic space $H^2(\mathbb{C})=SU(2,1)/S(U(2)\times U(1))$, according to \cite[Proposition 2]{Jensen}. 
On the other hand, it is well-known that $H^2(\mathbb{C})$ is an irreducible symmetric space (see \cite[pag. 315]{Besse}).

In the case of $(\mathfrak{d}_{4,2},\la \cdot,\cdot \ra_{t})$, 
it is not hard to see that $(\dd_{4,2},\langle \cdot,\cdot \rangle_{t})$ is isometric up to scaling to the metric Lie algebra described in $\ggo_{4,9}(\frac{1}{2})$, see \cite[Theorem 2.2]{Fino}.  It can be shown that  $(G,g)$ is isometric up to scaling to the unique irreducible proper 3-symmetric space of dimension four, according to \cite[Corollary 3.1]{Fino}. This space is also viewed as the irreducible K\"ahler surface corresponding to $\mathrm{F}_4$-geometry, see 
 \cite{Wall1} and \cite{Wall2}.  

\end{rem}

On the other hand we will analyze the case when  $(\ggo, \ip)$ belongs to Table \ref{tabla1}. 
First we recall the following theorem. 



\begin{thm}\cite[Theorem 10.3.2]{Petersen}
Let $(M, g)$  be an irreducible Riemannian manifold, and let $H$ be a parallel skew-symmetric $(1,1)$-tensor which is nowhere zero. Then $H=\lambda J$ for a complex parallel structure $J$ and $\lambda\in\mathbb{R}$.   
\end{thm}

\begin{prop}
Let $(\ggo,\la \cdot , \cdot \ra)$ be one of the metric Lie algebras of Table \ref{tabla1} and let $(G,g)$ be the associated simply connected  Lie group with the corresponding left invariant metric. Then $(G,g)$ is reducible as a Riemannian manifold, and after a suitable scaling $(G,g)$ is isometric to one of the following list: 
 $\mathbb{R}^4$, $\mathbb{R}^2\times\mathbb{H}^2$,
  or $\mathbb{H}^2\times\mathbb{H}^2$ (here the scaling is in each factor).

\end{prop}
\begin{proof} 
Clearly if $(\ggo, \ip)$ belongs to Table \ref{tabla1}, then $(G,g)$ is reducible. 
In the case $( \mathbb{R}\times \mathfrak{e}(2),\ip_t)$, one can easily check that the curvature tensor $R$ is zero. This implies that $\mathfrak{hol}(G,g)=0$ and hence any 1-dimensional subspace  of $\ggo=T_eG$ is invariant under $\hol(G,g)=\{e\}$, where $\{e\}$ is the trivial group.
Consequently, any decomposition of $T_eG$ as a direct sum of 1-dimensional subspaces leads to an identification of $G$ with $\mathbb{R}^4$ as Riemannian manifolds.

Before proceeding with the other cases, recall that the hyperbolic plane $\mathbb{H}^2=\{(x,y):y>0\}$ is a simply connected Lie group with multiplication $(x,y)\cdot (x',y')=(x+yx',yy')$, and that the usual metric $(ds)^2=\dfrac{(dx)^2+(dy)^2}{y^2}$ is left invariant. As a Riemannian manifold, it is irreducible with negative curvature.
 The associated metric Lie algebra is $\aff(\mathbb{R})=\Span\{e,f\}$, where $\{e,f\}$ is an orthonormal basis on which the Lie brackets are given by $[e,f]=f$.

With the previous observation, it is straightforward to prove that the cases $( \mathfrak{aff}(\mathbb{R})\times \mathbb{R}^2,\ip_t)$ and $( \mathfrak{aff}(\RR)\times \mathfrak{aff}(\RR),\ip_{t,s})$ corresponds to Riemannian manifolds $\mathbb{H}\times \mathbb{R}^2$ and $\mathbb{H}\times \mathbb{H}$ respectively.

Finally we consider the case $( \mathfrak{r}'_{4,\lambda,0},\ip_t)$. We can take the parameter $t=1$ in Table \ref{tabla1} since all the metrics are homothetic.
The commutator $[\ggo,\ggo]=\Span \{ e_2,f_2\}$ is abelian, the subspace $\Span \{ e_1,f_1\}$ with the induced metric is a metric Lie algebra isomorphic to $\aff(\mathbb{R})$ with the usual metric, and $\mathfrak{r}'_{4,\lambda,0}$ has the description $\aff(\mathbb{R})\ltimes \mathbb{R}^2$. Since the factors are orthogonal, we obtain a decomposition $G=\mathbb{H}^2\ltimes\mathbb{R}^2$, where $\mathbb{H}^2$ has the hyperbolic metric, $\mathbb{R}^2$ has the Euclidean metric and the factors are orthogonal. In other words, the de Rham decomposition of the pair $(G,g)$ associated with $(\ggo,\ip)$ is, up to scaling, $\mathbb{H}^2\times\mathbb{R}^2$.
\end{proof}
\

\end{document}